\documentclass[12pt]{amsart}

\usepackage{amsfonts}
\usepackage{amsmath}
\usepackage{amssymb}
\usepackage{amsthm}
\usepackage{graphicx}
\usepackage{color}
\usepackage{epsfig}
\usepackage{url}
\usepackage{placeins}

\numberwithin{equation}{section}
\theoremstyle{plain}
\newtheorem{thm}{Theorem}[section]
\newtheorem{rem}{Remark}[section]

\newtheorem{cor}{Corollary}[section]
\newtheorem{lem}{Lemma}[section]

\newcommand{\dE}{\mathbb{E}}
\newcommand{\dR}{\mathbb{R}}

\newcommand{\dP}{\mathbb{P}}

\newcommand{\cD}{\mathcal{D}}

\newcommand{\cM}{\mathcal{M}}

\newcommand{\wh}{\widehat}

\newcommand{\wt}{\widetilde}

\begin{document}
\title[Moderate deviations for parameters estimation in a  Heston process]
{Moderate deviations for parameters estimation in a geometrically ergodic Heston process \vspace{1ex}}
\author{Marie du ROY de CHAUMARAY}
\date \today
\address{Universit\'e de Bordeaux , Institut de Math\'{e}matiques de Bordeaux,
UMR 5251, 351 Cours de la Lib\'{e}ration, 33405 Talence cedex, France.}

\begin{abstract}
We establish a moderate deviation principle for the maximum likelihood estimator of the four parameters of a geometrically ergodic Heston process. We also obtain moderate deviations for the maximum likelihood estimator of the couple of dimensional and drift parameters of a generalized squared radial Ornstein-Uhlenbeck process. We restrict ourselves to the most tractable case where the dimensional parameter satisfies $a>2$ and the drift coefficient is such that$b<0$. In contrast to the previous literature, parameters are estimated simultaneously.
\end{abstract}
\maketitle

\section{Introduction}
In the recent theory of hedging, a particular attention has been drawn to the study of stochastic volatility models in which the volatility itself is given as a solution of some stochastic differential equation, see \cite{SteinStein}, \cite{Lewis} and \cite{Gath} for financial accuracy. Among them, Heston process \cite{hest} is one of the most popular, due to its computational tractability. For example, call option prices are succesfully computed in \cite{Lee} using Fourier inversion techniques. We denote by $Y_t$ the logarithm of the price of a given asset and by $X_t$ its instantaneous variance, and we consider the following Heston process
\begin{equation} \label{hm}
\left\lbrace 
\begin{array}{c @{\, = \, } l}
    \mathrm{d}X_t & (a+b X_t) \, \mathrm{d}t + 2 \sqrt{X_t}\, \mathrm{d}B_t  \\
    \mathrm{d}Y_t & (c+d X_t) \, \mathrm{d}t + 2 \sqrt{X_t}\, \left( \rho \,\mathrm{d}B_t + \sqrt{1-\rho^2} \, \mathrm{d}W_t \right) \\
\end{array}
\right.
\end{equation}
with $a >0$, $\left(b, c, d \right) \in \dR^3$ and $\rho \in ]-1,1[$, where $\left(B_t,W_t\right)$ is a 2-dimensional standard Wiener process and the initial state $\left(x_0,y_0\right) \in \dR^{+} \times \dR$.
In this process, the volatility $X_t$ is driven by a generalized squared radial Ornstein-Uhlenbeck process, also known as the CIR process, firstly studied by Feller \cite{fel} and introduced in a financial context by Cox, Ingersoll and Ross \cite{CIR} to compute short-term interest rates.
The behaviour of the CIR process has been widely investigated and depends on the values of both coefficients $a$ and $b$. We shall restrict ourself to the most tractable situation where $a>2$ and $b<0$. In this case, the CIR process is geometrically ergodic and never reaches zero. 

To calibrate the model, we estimate all the parameters $(a,b,c,d)$ at the same time using a trajectory of $(X_t)$ and $(Y_t)$ over the time interval $[0,T]$. Azencott and Gadhyan \cite{az} developed an algorithm to estimate some parameters of the Heston process based on discrete time observations, by making use of Euler and Milstein discretization schemes for the maximum likelihood. However, in the special case of an Heston process, the exact likelihood can be computed. It allows us to construct the maximum likelihood estimator (MLE) without using sophisticated approximation procedures, which are necessary for many stochastic volatility models, see \cite{AitK}. The MLE $\wh{\theta}_T=(\wh{a}_T,\wh{b}_T,\wh{c}_T,\wh{d}_T)$ of $\theta=\left(a,b,c,d\right)$  has been recently investigated in \cite{pap2}, together with its asymptotic behavior in the special case where $a \geq 2$.  It is given as follows
\begin{equation}\label{Est}\wh{\theta}_T=
\theta+ 2 \begin{pmatrix}
\left\langle M \right\rangle_T^{-1} & 0 \\
0 & \left\langle M \right\rangle_T^{-1}
\end{pmatrix}\begin{pmatrix}
M_T\\
N_T
\end{pmatrix},
\end{equation}
where $M_T$ and $N_T$ are continuous-time martingales respectively given by 
\begin{equation}\label{MN}
M_T=\left(\displaystyle \int_0^T{X_t^{-1/2}  \, \mathrm{d}B_t},\int_0^T{X_t^{1/2} \, \mathrm{d}B_t}\right) ^{\intercal}, N_T= \left(
\displaystyle \int_0^T{X_t^{-1/2} \, \mathrm{d}\wt{B}_t},
\displaystyle \int_0^T{X_t^{1/2} \, \mathrm{d}\wt{B}_t}\right)^{\intercal}
\end{equation}
with $\mathrm{d}\wt{B}_t= \rho \,\mathrm{d}B_t + \sqrt{1-\rho^2} \, \mathrm{d}W_t$, and $\left\langle M \right\rangle_T$ is the increasing process of $M_T$ given by
\begin{equation}\label{crochet1}
\left\langle M \right\rangle_T= T\begin{pmatrix}
\Sigma_T & 1\\
1 & S_T
\end{pmatrix}
\end{equation}
with $S_T= T^{-1}\int_0^T{X_t  \, \mathrm{d}t}$ and $\Sigma_T= T^{-1} \int_0^T{X_t^{-1}  \, \mathrm{d}t}$.
This estimator is strongly consistant, i.e. $\wh{\theta}_T$ converges almost surely to $\theta$ as $T$ goes to infinity. It also satisfies the following central limit theorem (CLT):
$$\sqrt{T}\begin{pmatrix}\wh{\theta}_T-\theta \end{pmatrix} \xrightarrow{\mathcal{L}} \mathcal{N}(0,4\Gamma^{-1}) $$
where the block matrix $\Gamma$ is given by
\begin{equation}\label{defSigma}
\Gamma= R \otimes \Sigma = \begin{pmatrix}
\Sigma & \rho \Sigma\\
\rho \Sigma & \Sigma
\end{pmatrix} \hspace{1cm} \text{with} \hspace{1cm}  \Sigma= \begin{pmatrix}
 \frac{-b}{a-2} & 1 \\
 1 & -\frac{a}{b}
 \end{pmatrix} \, \text{ and } \, R=\begin{pmatrix}
1 & \rho\\
\rho & 1
\end{pmatrix} 
 \end{equation}
 where $\otimes$ stands for the Kronecker product.
One can observe that $(\wh{a}_T,\wh{b}_T)$ coincides with the MLE of the parameters $(a,b)$ of the CIR process based on the observation of $(X_T)$ over the time interval $\left[0,T\right]$: 
\begin{equation}\label{estCir}
\begin{pmatrix}
\wh{a}_T\\
\wh{b}_T
\end{pmatrix}= \begin{pmatrix}
a\\
b
\end{pmatrix}+ 2 \left\langle M\right\rangle_T^{-1} M_T.
\end{equation}
Asymptotic results about this estimator can be found in \cite{Ov}, \cite{FT} and \cite{KAB2}, and a large deviation principle (LDP) has been recently established in \cite{dRdC1}.
In the easier case where one parameter is estimated while the other one is supposed to be known, \cite{Z} gives large deviations whereas \cite{GJ2} derives moderate deviations.

 In this paper, our goal is to establish a moderate deviation principle (MDP) for the MLE of the four parameters of the Heston process given in \eqref{Est}, which is an natural continuation of the central limit theorem. Let us first recall some basic definitions of large deviation theory. We refer to \cite{DeZ} for further details. Let $(\lambda_T)_T$ be a positive sequence of real numbers increasing to infinity with $T$. A sequence $\left(Z_T\right)_T$ of $\dR^d$-valued random variables satisfies a large deviation principle (LDP) with speed $\lambda_T$ and rate function $I:\dR^d \mapsto [0,+\infty]$ if $I$ is lower semi-continuous and such that $\left(Z_T\right)_T$ satisfies the following upper and lower bounds: for any closed set $F$ of $\dR^d$
$$\limsup_{T \to +\infty} \lambda_T^{-1} \log \dP \left(Z_T \in F\right) \leq - \inf_{z\in F}{I\left(z\right)}$$
and for any open set $G$ of $\dR^d$
$$\liminf_{T \to +\infty} \lambda_T^{-1}  \log \dP \left(Z_T \in G\right) \geq - \inf_{z\in G}{I\left(z\right)}.$$
If furthermore the level sets of $I$ are compacts, $I$ is called a good rate function.
Additionally, if $\lambda_T = o(T)$, a sequence $\left(Z_T\right)_T$ of $\dR^d$-valued random variables satisfies a moderate deviation principle (MDP) with speed $\lambda_T$ and rate function $I$, if the sequence $(\sqrt{T/\lambda_T} Z_T)_T$ satisfies an LDP with speed $\lambda_T$ and rate function $I$.

In other words, let $\left(\lambda_T\right)_T$ be a positive sequence satisfying for $T$ going to infinity
\begin{equation}\label{E0}
\lambda_T \rightarrow +\infty \: \: \text{ and } \: \: \frac{\lambda_T}{T} \rightarrow 0.
\end{equation}
We investigate in this paper an LDP for $\sqrt{\frac{T}{\lambda_T}}(\wh{\theta}_T-\theta)$ with speed $\lambda_T$ and compute the explicit rate function. By the way, we also establish an MDP for the MLE of the two parameters of a CIR process. The paper is organised as follows. Section 2 contains our main results while Sections 3 and 4 are devoted to their proofs.

\section{Main results}
We establish MDPs for the MLE $(\wh{a}_T,\wh{b}_T)$ of the parameters $(a,b)$ of the CIR process given by \eqref{estCir}, as well as for the MLE $\wh{\theta}_T$ of the Heston process given by \eqref{Est}. 

\begin{thm}\label{MDPCIR}
The sequence $\left(\sqrt{\frac{T}{\lambda_T}}(\wh{a}_T-a,\wh{b}_T-b)\right)$ satisfies an LDP with speed $\lambda_T$ and good rate function $I_{a,b}$ given for all $(\alpha,\beta)\in \mathbb{R}^2$ by
\begin{equation}\label{Iab}
I_{a,b}(\alpha,\beta)=-\frac{b}{8(a-2)}\alpha^2-\frac{a}{8b}\beta^2+\frac{\alpha \beta}{4}.
\end{equation}
\end{thm}

\begin{thm}\label{MDPH}
The sequence $\left(\sqrt{\frac{T}{\lambda_T}}(\wh{\theta}_T-\theta)\right)$ satisfies an LDP with speed $\lambda_T$ and good rate function $I_{\theta}$ given for all $(\alpha,\beta,\gamma,\delta)\in \mathbb{R}^4$ by
\begin{equation*}
I_{\theta}(\alpha,\beta, \gamma, \delta)=\left(1-\rho^2\right)^{-1} \left( I_{a,b}\left(\alpha,\beta\right) + I_{a,b}\left(\gamma,\delta\right) + \rho  \, J\left(\alpha,\beta,\gamma,\delta\right)\right).
\end{equation*}
where $I_{a,b}$ is given by \eqref{Iab} and $$J\left(\alpha,\beta,\gamma,\delta\right) =-\alpha \delta + \frac{b}{a-2} \alpha \gamma - \beta \gamma+ \frac{a}{b} \beta \delta.$$
\end{thm}

\begin{rem}
One can observe that the first result is a particular case of the second one, with $\gamma=\delta=\rho=0$, and we easily check that $$I_{a,b}=I_{\theta} \left(\cdot, \cdot , 0 ,0\right).$$
\end{rem}
The proofs are respectively postponed to Sections 3 and 4.

\section{MDP for the CIR process}
We rewrite the MLE $(\wh{a}_T,\wh{b}_T)$ of the parameters $(a,b)$ of the CIR process as follows:
\begin{equation}\label{EstCir}
\sqrt{\frac{T}{\lambda_T}}\begin{pmatrix}
\wh{a}_T-a\\
\wh{b}_T-b\\
\end{pmatrix} = 2 T \left\langle M \right\rangle_T^{-1} \frac{1}{\sqrt{\lambda_T T}}M_T.
\end{equation}
In order to prove Theorem \ref{MDPCIR}, we first establish LDPs with speed $\lambda_T$ for $\left(\lambda_T T\right)^{-1/2}M_T$ and for $T^{-1}\left\langle M \right\rangle_T$. Then we conclude using the contraction principle (see Theorem 4.2.1 of \cite{DeZ}) which is recalled here for completeness.

\begin{lem}[Contraction Principle]\label{CP}
Let $\left(Z_T\right)_T$ be a sequence of random variables of $\dR^d$ satisfying an LDP with good rate function $I$ and $g:\dR^d \to \dR^n$ be a continuous function over $\mathcal{D}_{I}=\left\lbrace x \in \dR^d | I(x) <+\infty\right\rbrace$. Then, the sequence $\left(g(Z_T)\right)_T$ satisfies an LDP with good rate function $J$ defined for all  $y \in \dR^n$ by
$$J\left(y\right)= \underset{\left\lbrace x \, \in \cD_I | g(x)=y \right\rbrace}{\inf} I\left(x\right),$$ 
where the infimum over the empty set is equal to infinity.
\end{lem} 

\begin{lem}\label{MPDM}
The sequence $\left(\left(\lambda_T T\right)^{-1/2}M_T\right)$ satisfies an LDP with speed $\lambda_T$ and good rate function $I_M$ given, for all $(m,n) \in \mathbb{R}^2$, by 
\begin{equation}\label{Im}
I_M(m,n)=-\frac{a(a-2)}{4b}m^2-\frac{b}{4}n^2-\frac{a-2}{2}mn.
\end{equation}
\end{lem}

\begin{proof}
We establish this LDP by applying the G\"artner-Ellis Theorem. Thus, for $T$ going to infinity, we need to compute the pointwise limit $\Lambda$ of the normalized cumulant generating function $\Lambda_T$ of $\left(\left(\lambda_T T\right)^{-1/2}M_T\right)$ given, for all $v \in \dR^2$, by
\begin{equation*}
\Lambda_T\left(v\right)= \frac{1}{\lambda_T} \log \dE \left[ e^{\sqrt{\lambda_T/T} \, \left\langle v, M_T\right\rangle}\right]
\end{equation*} We show in Corollary B.1 of Appendix B that for all $v \in \dR^{2}$,
  \begin{equation}\label{limncgf}
\Lambda(v)=\lim_{T \to +\infty} \Lambda_T \left(v\right)= \frac{1}{2} \,  v^{\intercal} \,  \Sigma \, v, 
\end{equation} where the matrix $\Sigma$ was previously given in \eqref{defSigma}.
Thus, by the G\"artner-Ellis Theorem, the sequence $\left(\left(\lambda_T T\right)^{-1/2}M_T\right)$ satisfies an LDP with speed $\lambda_T$ and rate function $I_M$ given by the Fenchel-Legendre transform of $\Lambda$, for all $\mu \in \dR^{2}$,
\begin{equation}\label{IM}\begin{array}{ll}
I_M(\mu)& = \displaystyle \sup \left\lbrace \left\langle \mu,v\right\rangle - \Lambda(v) \, | \, v \in \mathbb{R}^2 \right\rbrace\\
& = \displaystyle \sup \left\lbrace \mu^{\intercal}v - \frac{1}{2} \,  v^{\intercal} \,  \Sigma \, v  \, | \, v \in \mathbb{R}^2 \right\rbrace.
\end{array}
\end{equation}
The function of $v$ that we need to optimise is non convex with a critical point $v_0= \Sigma^{-1} \mu$. Replacing it into \eqref{IM} and using the fact that $\Sigma$ is symmetric, we obtain that 
\begin{equation*}
I_M(\mu)= \frac{1}{2} \,  \mu^{\intercal} \,  \Sigma^{-1} \, \mu
\end{equation*}
which easily leads to the announced result.
\end{proof}

\begin{lem}\label{MDPN}
The sequence $\left(\left(\lambda_T T\right)^{-1/2} N_T\right)$ satisfies an LDP with speed $\lambda_T$ and good rate function $I_M$ given by \eqref{Im}. 
\end{lem}

\begin{proof}
It works as in the previous proof, except that, this time, we make use of Corollary B.2.
\end{proof}

\begin{lem}\label{Conti}
The sequences $\left(2 T \left\langle M \right\rangle_T^{-1} \left(\lambda_T T\right)^{-1/2} M_T\right)$ and $\left(2 \Sigma^{-1} \left(\lambda_T T\right)^{-1/2} M_T\right)$ are exponentially equivalent with speed $\lambda_T$, which means that, for all $\varepsilon >0$,
\begin{equation*}
\limsup_{T \rightarrow +\infty} \frac{1}{\lambda_T} \log \dP\left( \parallel  \left(T \left\langle M \right\rangle_T^{-1}-\Sigma^{-1}\right) \left(\lambda_T T\right)^{-1/2}M_T \parallel \geq \varepsilon \right)=-\infty,
\end{equation*}
where $\| \cdot \|$ is the Euclidean norm.
\end{lem}

\begin{proof}
For all $\varepsilon >0$ and all $\eta >0$, we have the following upper bound:
$$\limsup_{T \rightarrow +\infty} \frac{1}{\lambda_T} \log \dP\left( \parallel  \left(T \left\langle M \right\rangle_T^{-1}-\Sigma^{-1}\right) \left(\lambda_T T\right)^{-1/2} M_T \parallel \geq \varepsilon \right) \leq \limsup_{T \rightarrow +\infty} \max \left\lbrace P_T^{\eta} ; Q_T^{\eta,\varepsilon} \right\rbrace$$
with $$P_T^{\eta}= \frac{1}{\lambda_T} \log \dP\left( \parallel  T \left\langle M \right\rangle_T^{-1}-\Sigma^{-1} \parallel \geq \eta \right)$$ and $$ Q_T^{\eta,\varepsilon}=\frac{1}{\lambda_T} \log \dP\left( \parallel \left(\lambda_T T\right)^{-1/2} M_T \parallel \geq \frac{\varepsilon}{\eta} \right),$$ where we also denote by $\| \cdot \|$ the subordinate norm.
On the one hand, we rewrite
\begin{align*}P_T^{\eta} &= \frac{T}{\lambda_T} \frac{1}{T} \log \dP \left( \|T^{-1}\left\langle M \right\rangle_T^{-1}-\Sigma^{-1}\| \geq \eta \right)\\
& \leq \frac{T}{\lambda_T} \frac{1}{T} \log \dP \left( \left(\frac{\Sigma_T}{V_T}+\frac{b}{2}\right)^2+ \left(\frac{S_T}{V_T}+\frac{a(a-2)}{2b}\right)^2  \geq \eta^2 \right)
\end{align*}
where $V_T=S_T\Sigma_T-1$, and we show in Appendix C that, for all $\eta >0$,
\begin{equation}\label{Peta}
p^{\eta}:=\limsup_{T \rightarrow +\infty} \frac{1}{T} \log \dP \left( \left(\frac{\Sigma_T}{V_T}+\frac{b}{2}\right)^2+ \left(\frac{S_T}{V_T}+\frac{a(a-2)}{2b}\right)^2  \geq \eta^2 \right) <0.
\end{equation}
Using the fact that $\frac{T}{\lambda_T}$ tends to infinity, we obtain that for all $\eta >0$,
$$\limsup_{T \rightarrow +\infty} P_T^\eta = -\infty. $$
On the other hand, as we have established in Lemma \ref{MPDM} an MDP for $\left(\left(\lambda_T T\right)^{-1/2}M_T\right)$ with rate function  
$I_M$, we have the following upper bound
\begin{equation*}
\limsup_{T \rightarrow +\infty} Q_T^{\eta,\varepsilon} \leq - \inf \left\lbrace I_M\left(m,n\right) | \left(m,n\right) \notin \mathcal{B} \left(0, \varepsilon/\eta\right)\right\rbrace
\end{equation*}
Letting $\eta$ tend to zero, we obtain that
\begin{equation*}
\limsup_{\eta \rightarrow 0} \limsup_{T \rightarrow +\infty} Q_T^{\eta,\varepsilon} = -\infty.
\end{equation*}
Thus, for all $\varepsilon >0$,
\begin{equation*}
\limsup_{\eta \rightarrow 0} \limsup_{T \rightarrow +\infty} \max \left\lbrace P_T^{\eta} ; Q_T^{\eta,\varepsilon} \right\rbrace = -\infty.
\end{equation*}
\end{proof}

\begin{lem}\label{contigN}
The sequences $\left(2 T \left\langle M \right\rangle_T^{-1} \left(\lambda_T T\right)^{-1/2} N_T\right)$ and $\left(2 \Sigma^{-1} \left(\lambda_T T\right)^{-1/2} N_T\right)$ are exponentially equivalent with speed $\lambda_T$, which means that, for all $\varepsilon >0$,
\begin{equation*}
\limsup_{T \rightarrow +\infty} \frac{1}{\lambda_T} \log \dP\left( \parallel  \left(T \left\langle N \right\rangle_T^{-1}-\Sigma^{-1}\right) \left(\lambda_T T\right)^{-1/2}N_T \parallel > \varepsilon \right)=-\infty,
\end{equation*}
where $\| \cdot \|$ is the Euclidean norm.
\end{lem}

\begin{proof}
As the martingales $M_T$ and $N_T$ share the same bracket process and the same MDP, the previous proof remains valid for the martingale $N_T$ instead of $M_T$.
\end{proof}

 Theorem \ref{MDPCIR} immediately follows from Lemmas \ref{MPDM} and \ref{Conti}, by a straightforward application of the contraction principle recalled in Lemma \ref{CP}.

\begin{proof}[Proof of Theorem \ref{MDPCIR}]
 Combining Lemma \ref{CP} with Lemma \ref{MPDM}, we show that the sequence $(2 \Sigma^{-1}\left(\lambda_T T\right)^{-1/2}M_T)$ satisfies an LDP with speed $\lambda_T$ and rate function $I_{a,b}$ given for all $\left(\alpha,\beta\right) \in \dR^2$ by $$I_{a,b}(\alpha,\beta)=I_M\left(\frac{1}{2} \, \Sigma \, (\alpha,\beta)^{\intercal} \right),$$ where $I_M$ is obtained in Lemma \ref{MPDM}. 
Besides, by Lemma \ref{Conti} and equation \eqref{EstCir}, the sequences $\sqrt{T\lambda_T^{-1}}\left(\wh{a}_T-a,\wh{b}_T-b\right)$ and $(2 \Sigma^{-1}\left(\lambda_T T\right)^{-1/2}M_T)$ share the same LDP with speed $\lambda_T$. This gives the announced result.
\end{proof}

\section{MDP for the Heston model given by \eqref{hm}}
We now go back to the MLE $\wh{\theta}_T$ and prove the main Theorem \ref{MDPH}. We denote by $\cM_T$ the continuous-time martingale $$\cM_T=\begin{pmatrix}
M_T\\
N_T
\end{pmatrix},$$ where $M_T$ and $N_T$ are given by \eqref{MN}.
As $\left\langle \mathrm{d}B_t, \mathrm{d}\wt{B}_t \right\rangle = \rho \,  \mathrm{d}t$, we easily obtain that the bracket process of $\cM_T$ equals
\begin{equation}\label{crochet}
\left\langle \cM\right\rangle_T = R \otimes \left\langle M \right\rangle_T  
\end{equation}
where the matrix $R$ is given in \eqref{defSigma}.
We follow the same scheme than in the previous section. At first, we establish an LDP with speed $\lambda_T$ for the $4$-dimensional martingale $\left(\lambda_T T\right)^{-1/2} \cM_T$, by making use of the exponential convergence of its bracket process. Then we show that the estimator is exponentially equivalent to some sequence involving the matrix $\Sigma$ and the martingale $\cM_T$, for which we are able to establish an LDP by applying the contraction principle.

\begin{lem}\label{MDPm}
The sequence $\left(\left(\lambda_T T\right)^{-1/2}\cM_T\right)$ satisfies an LDP with speed $\lambda_T$ and good rate function $I_{\cM}$ given for all $(x,y,z,t) \in \mathbb{R}^4$ by 
\begin{equation*}
I_{\cM}(x,y,z,t)=\left(1-\rho^2\right)^{-1} \left(I_M(x,y)+I_M(z,t)+\rho \frac{a-2}{2}\left(yz+xt+\frac{a}{b}zx+\frac{b}{a-2}ty\right) \right),
\end{equation*}
where $I_M$ is given by \eqref{Im}.
\end{lem}

\begin{proof}We establish this LDP by applying again the G\"artner-Ellis Theorem. Thus, we need to compute the pointwise limit, for $T$ going to infinity, of the normalized cumulant generating function $\mathcal{L}_T$ of $\left(\left(\lambda_T T\right)^{-1/2}\cM_T\right)$ given for all $u \in \dR^4$ by
\begin{equation*}\label{ncgf2}
\mathcal{L}_T\left(u\right)= \frac{1}{\lambda_T} \log \dE \left[ e^{\sqrt{\lambda_T/T} \, \left\langle u, \mathcal{M}_T\right\rangle}\right].
\end{equation*}
We show in Appendix B that,
  \begin{equation}\label{limncgf2}
\lim_{T \to +\infty} \mathcal{L}_T \left(u\right)= \frac{1}{2} u^{T} \Gamma u,
\end{equation} where the matrix $\Gamma$ is given in \eqref{defSigma}. 
Thus, the sequence $\left(\left(\lambda_T T\right)^{-1/2}\cM_T\right)$ satisfies an LDP with speed $\lambda_T$ and good rate function $I_{\cM}$ given for all $\mu \in \mathbb{R}^4$ by 
\begin{equation}\label{IcM}
I_{\cM}\left(\mu\right) = \sup_{\Lambda \in \dR^4} \left\lbrace  \mu^{\intercal}\Lambda- \frac{1}{2} \,  \Lambda^{\intercal} \Gamma \Lambda \right\rbrace = \frac{1}{2} \,  \mu^{\intercal} \Gamma^{-1} \mu.
\end{equation}
It is not hard to see that
\begin{equation}\label{invQ}
\Gamma^{-1}= R^{-1} \otimes \Sigma^{-1} =\left(1-\rho^2\right)^{-1} \begin{pmatrix}
\Sigma^{-1} & -\rho \Sigma^{-1}  \\
-\rho \Sigma^{-1} & \Sigma^{-1}
\end{pmatrix}.
\end{equation} 
We complete the proof of Lemma \ref{MDPm} combining \eqref{IcM} and \eqref{invQ}.
\end{proof}

\begin{lem}\label{contig2}
We denote by $\mathcal{S}$ the block matrix
\begin{equation}\label{QQ}
\mathcal{S}=I_2 \otimes \Sigma =\begin{pmatrix}
\Sigma & 0 \\
0 & \Sigma
\end{pmatrix}
\end{equation}
where $I_2$ stands for the identity matrix of size 2 and $\Sigma$ is given by \eqref{defSigma}.
The sequences $\left(\sqrt{\frac{T}{\lambda_T}} (\wh{\theta}_T-\theta)\right)$ and $\left(2 \mathcal{S}^{-1} \left(\lambda_T T\right)^{-1/2}\cM_T\right)$ are exponentially equivalent with speed $\lambda_T$.

\end{lem}

\begin{proof}
We want to prove that for all $\varepsilon >0$,
\begin{equation}\label{objectif}
\limsup_{T \rightarrow +\infty} \lambda_T^{-1} \log \dP\left( \parallel (I_2 \otimes B_T) \left(\lambda_T T\right)^{-1/2}\cM_T \parallel > \varepsilon \right)=-\infty,
\end{equation}
where $\| \cdot \|$ is the Euclidean norm and $B_T=T\left\langle M \right\rangle_T^{-1}-\Sigma^{-1}$.
We have the following upper bound:
\begin{equation*}
\dP\left( \parallel (I_2 \otimes B_T)\left(\lambda_T T\right)^{-1/2}\cM_T \parallel > \varepsilon \right) \leq P_T^M + P_T^N
\end{equation*}
where $P_T^M=
\dP\left( \parallel B_T \left(\lambda_T T\right)^{-1/2}M_T \parallel > \frac{\varepsilon}{\sqrt{2}} \right)$ and $P_T^N= \dP\left( \parallel B_T \left(\lambda_T T\right)^{-1/2}N_T \parallel > \frac{\varepsilon}{\sqrt{2}} \right)$.
We have shown in Lemma \ref{Conti} that 
$$\limsup_{T \rightarrow +\infty} \lambda_T^{-1} \log P_M = -\infty,$$
and, by Lemma \ref{contigN}, we also know that
$$\limsup_{T \rightarrow +\infty} \lambda_T^{-1} \log P_N = -\infty.$$
This leads to \eqref{objectif} which gives the announced result.
\end{proof}

The proof of Theorem \ref{MDPH} is a direct consequence of both Lemmas \ref{MDPm} and \ref{contig2}.
 
\begin{proof}[Proof of Theorem \ref{MDPH}]
The contraction principle together with Lemma \ref{MDPm} imply that the sequence $\left(2 \mathcal{S}^{-1} \left(\lambda_T T\right)^{-1/2}\cM_T\right)$ satisfies an LDP with speed $\lambda_T$. However, it follows from Lemma \ref{contig2} that  this sequence is exponentially equivalent to $\sqrt{\frac{T}{\lambda_T}}\left(\wh{\theta}_T-\theta\right)$, which means that they share the same LDP. The rate function $I_{\theta}$ is given by the contraction principle as follows. For all $\mu \in \dR^4$,
\begin{equation*}
I_{\theta}\left(\mu\right)= \inf_{\Lambda \in \dR^4}\left\lbrace I_{\cM}(\Lambda) \, | \,  \mu = 2 \mathcal{S}^{-1} \Lambda \right\rbrace
\end{equation*}
where $I_{\cM}$ is given in Lemma \ref{MDPm}. Thus, 
\begin{equation*}
I_{\theta}\left(\mu\right)= I_{\cM} \left( \frac{1}{2} \, \mathcal{S} \mu \right)
\end{equation*}
which leads to the announced result, just replacing $I_{\cM}$ by its expression.
\end{proof}

\section*{Appendix A: Changes of parameters}
\renewcommand{\thesection}{\Alph{section}}
\renewcommand{\theequation}{\thesection.\arabic{equation}}
\setcounter{section}{1}
\setcounter{equation}{0}
To compute the limits of the cumulant generating functions \eqref{limncgf} and \eqref{limncgf2}, which is the aim of Appendix B, we recall some changes of probability formulas. We denote by $\dP_{c,d}^{a,b}$ the distribution of the solution of \eqref{hm} associated with parameters $a$, $b$, $c$ and $d$, and by $\dE_{c,d}^{a,b}$ the corresponding expectation. At first, we change both parameters $a$ and $b$ of the first equation of \eqref{hm}, which corresponds to the CIR process. Applying Girsanov's formula given e.g. in Theorem 1.12 of \cite{KU}, we have 
\begin{equation}\label{change1}
\begin{array}{lcl}
 \displaystyle \frac{\mathrm{d}\dP_{c,d}^{a,b}}{\mathrm{d}\dP_{c,d}^{\alpha,\beta}}&=& \displaystyle X_T^{\frac{a-\alpha}{4}} \exp \left(-\frac{a-\alpha}{4} \left(\log X_0 +bT\right) + \frac{b-\beta}{4}(X_T-X_0-\alpha T) \right) \\
 & {} & \, \displaystyle \times \exp \left( - \frac{T}{8} \left((b^2-\beta^2)  S_T - \left(4 (\alpha-a) - \alpha^2 + a^2\right) \Sigma_T\right) \right).
 \end{array}
\end{equation}
We also need to change parameters $c$ and $d$ of the second equation in \eqref{hm}. We rewrite this equation with new parameters $\gamma$ and $\delta$:

\begin{equation*}
\mathrm{d}Y_t=(\gamma+\delta X_t) \, \mathrm{d}t + 2 \sqrt{X_t}\,\left[ \rho \,\mathrm{d}B_t + \sqrt{1-\rho^2} \mathrm{d}\wt{W}_t\right]
\end{equation*}
where 
\begin{equation*}
\mathrm{d}\wt{W}_t=\mathrm{d}W_t + \frac{\mathrm{d}t}{2\sqrt{1-\rho^2}} \left(\left(c-\gamma\right) X_t^{-1/2} + \left(d-\delta\right) X_t^{1/2} \right).
\end{equation*}
Thus, we obtain that
\begin{equation}\label{change2}
\begin{array}{lcl}
 \displaystyle \frac{\mathrm{d}\dP_{c,d}^{a,b}}{\mathrm{d}\dP_{\gamma,\delta}^{a,b}}&=& \displaystyle \exp \left(\frac{c-\gamma}{2\sqrt{1-\rho^2}} \int_0^T X_t^{-1/2} \, \mathrm{d}W_t+\frac{d-\delta}{2\sqrt{1-\rho^2}} \int_0^T X_t^{1/2} \, \mathrm{d}W_t\right) \\
 & {} & \, \displaystyle  \times \exp \left(\frac{T}{8\left(1-\rho^2\right)}\left[ (d-\delta)^2  S_T + \left(c-\gamma\right)^2  \Sigma_T + 2 \left(c-\gamma\right)\left(d-\delta\right) \right]\right).
 \end{array}
\end{equation}


\section*{Appendix B: Proof of the pointwise limit of the cumulant generating function }
\renewcommand{\thesection}{\Alph{section}}
\renewcommand{\theequation}{\thesection.\arabic{equation}}
\setcounter{section}{2}
\setcounter{equation}{0}

We want to compute the pointwise limit of the cumulant generating function $\mathcal{L}_T$ of $\left(\lambda_T T\right)^{-1/2} \cM_T$. With $\cM_T$ being replaced by its expression in $X_t$, $B_t$, $W_t$ and $\rho$, we have that, for all $u=\left(u_1,u_2,u_3,u_4\right) \in \dR^4$,
\begin{equation}\label{laplace}
E_T\left(u\right):=\dE_{c,d}^{a,b} \left[ e^{\sqrt{\lambda_T/T} \, \left\langle u, \mathcal{M}_T\right\rangle}\right]= \dE_{c,d}^{a,b} \left[\mathcal{E}_{1,T} \,  \mathcal{E}_{2,T} \, \mathcal{E}_{3,T} \, \mathcal{E}_{4,T}\right]
\end{equation} 
where 
\begin{equation*}
\mathcal{E}_{1,T}= \exp \left(v_{3,T} \, \int_0^T X_t^{-1/2} \, \mathrm{d}B_t \right),
\, \, \, \,
\mathcal{E}_{2,T}= \exp \left(v_{4,T} \int_0^T X_t^{1/2} \, \mathrm{d}B_t \right),
\end{equation*}

\begin{equation*}
\mathcal{E}_{3,T}= \exp \left(\sqrt{\frac{\lambda_T}{T}} \sqrt{1-\rho^2} u_3 \int_0^T X_t^{-1/2} \, \mathrm{d}W_t \right),
\end{equation*}

\begin{equation*}
\mathcal{E}_{4,T}= \exp \left(\sqrt{\frac{\lambda_T}{T}} \sqrt{1-\rho^2} u_4 \int_0^T X_t^{1/2} \, \mathrm{d}W_t \right).
\end{equation*}
with $v_{3,T}=\sqrt{\lambda_T/T}(u_1 + \rho u_3)$ and $v_{4,T}=\sqrt{\lambda_T/T}(u_2+ \rho u_4)$.
We use \eqref{change2} to change parameters $c$ and $d$ in order to kill the terms involving $W_t$. We obtain that
\begin{equation}\label{EE1}
\dE_{c,d}^{a,b} \left[\mathcal{E}_{1,T} \,  \mathcal{E}_{2,T} \, \mathcal{E}_{3,T} \, \mathcal{E}_{4,T}\right]= \dE_{\gamma_T,\delta_T}^{a,b}\left[G_T \mathcal{E}_{1,T} \,  \mathcal{E}_{2,T} \right]
\end{equation} 
where $\gamma_T= c+ 2 u_3 \sqrt{\lambda_T/T} \left(1-\rho^2\right)$, $\delta_T= d + 2 u_4 \sqrt{\lambda_T/T} \left(1-\rho^2\right)$ and 
\begin{equation}\label{GT}
G_T=\exp\left(\frac{T}{8\, (1-\rho^2)}\left((c-\gamma_T)^2 \Sigma_T+(d-\delta_T)^2 S_T + 2 (c-\gamma)(d-\delta)\right) \right).
\end{equation}
Additionally, using the first equation of \eqref{hm} and Ito's formula applied to $\log X_T$, we obtain that
\begin{equation}\label{B1}
2 \displaystyle \int_0^T{X_t^{1/2}  \, \mathrm{d}B_t} = X_T - x_0 -a \, T -b \, T S_T
\end{equation}
and \begin{equation}\label{B2}
2\displaystyle \int_0^T{X_t^{-1/2} \, \mathrm{d}B_t} = \log X_T - \log x_0 - b \, T +(2-a) \, T \, \Sigma_T
\end{equation}
where $S_T$ and $\Sigma_T$ are given after \eqref{crochet1}. Thus, we rewrite $\mathcal{E}_{1,T}$ and $\mathcal{E}_{2,T}$ as functions of $X_T$, $S_T$ and $\Sigma_T$:\begin{equation}\label{E2T}
\mathcal{E}_{1,T}= \exp \left( \frac{v_{3,T}}{2}  \left(\log X_T- \log x_0-bT+ (2-a) T \Sigma_T\right)\right)
\end{equation}
and
\begin{equation}\label{E1T}
\mathcal{E}_{2,T}= \exp \left( \frac{1}{2} v_{4,T} \left(X_T-x_0-aT-bTS_T\right)\right).
\end{equation}

Thus, replacing \eqref{GT}, \eqref{E1T} and \eqref{E2T} into \eqref{EE1}, and taking out of the expectation all the deterministic terms, we obtain that
\begin{equation}\label{EE2}
E_T\left(u\right)=\exp \left(\mathcal{A}_T\right) \dE_{\gamma_T,\delta_T}^{a,b}\left[\exp \left(v_{1,T} T \Sigma_T + v_{2,T} T S_T +\frac{1}{2}  \left(v_{3,T} X_T + v_{4,T} \log X_T \right)\right)\right]
\end{equation}
where 
\begin{equation*}
v_{1,T}=\frac{\left(c-\gamma_T\right)^2}{8\left(1-\rho^2\right)}+ \frac{2-a}{2}v_{3,T} 
\, , \, \, \, 
v_{2,T}=\frac{\left(d-\delta_T\right)^2}{8\left(1-\rho^2\right)}- \frac{b}{2}v_{4,T}
\end{equation*}
and
\begin{equation*}
\mathcal{A}_T =T \left(1-\rho^2\right)^{-1} \left(c-\gamma_T\right)\left(d-\delta_T\right)/4
- v_{4,T} \left(aT+x_0\right)/2- v_{3,T} \left( \log x_0 + bT\right)/2. 
\end{equation*}
We now make a new change of parameters for $a$ and $b$, in order to kill the terms involving $S_T$ and $\Sigma_T$. The new time-depending parameters are given, for $T$ large enough, by 
\begin{equation*}
\alpha_T =2+\left(a-2\right)\sqrt{1 - 8\, v_{1,T}/\left(a-2\right)^2} \; \; \text{ and } \; \;  \beta_T=b\sqrt{1-8 \, v_{2,T}/b^2}.
\end{equation*}
Thus, \eqref{EE2} becomes 
\begin{equation}\label{EE3}
E_T\left(u\right)=\exp \left(-w_{1,T} x_0 - w_{2,T} \log x_0 + T \mathcal{C}_T\right) \dE_{\gamma_T, \delta_T}^{\alpha_T,\beta_T}\left( \exp \left(w_{1,T} X_T\right)X_T ^{w_{2,T}}\right)
\end{equation}
where $w_{1,T} =\left(b-\beta_T + 2 v_{4,T}\right)/4 $, $w_{2,T} =\left(a-\alpha_T + 2 v_{3,T}\right)/4 $
and $$4 \mathcal{C}_T=- 2 a v_{4,T} - 2 b v_{3,T}+ \beta_T \alpha_T - ab - (c-\gamma_T) (d-\delta_T)/(1-\rho^2).$$
Therefore, taking the logarithm of \eqref{EE3}, we obtain that
\begin{equation}\label{EE4}
 \log E_T\left(u\right)= -w_{1,T} x_0 - w_{2,T} \log x_0 + T \mathcal{C}_T+ \log \dE_{\gamma_T, \delta_T}^{\alpha_T,\beta_T}\left( \exp \left(w_{1,T} X_T\right)X_T ^{w_{2,T}}\right).
\end{equation}
We now have to divide \eqref{EE4} by $\lambda_T$ and investigate the limit for $T$ going to infinity. We consider each term of the right-hand side of \eqref{EE4} separately.
First of all, as $w_{1,T}$ and $w_{2,T}$ tend to zero for $T$ going to infinity, we immediately deduce that
\begin{equation}\label{Lim1}
\lim_{T \to +\infty} \lambda_T^{-1} \left(-w_{1,T} x_0 - w_{2,T} \log x_0 \right)=0.
\end{equation}
We now consider the term $T\mathcal{C}_T/ \lambda_T$.
On the one hand, replacing $\gamma_T$ and $\delta_T$ by their respective definitions, we easily obtain that 
\begin{equation}\label{l1}
\frac{1}{4 \left(1-\rho^2\right)} \left(c-\gamma_T\right)\left(d-\delta_T\right)\frac{T}{\lambda_T}= \left(1-\rho^2\right) u_3 u_4.
\end{equation}
On the other hand, as $\lambda_T/T$ goes to zero for $T$ tending to infinity, we expand $\beta_T \alpha_T$ up to order two in $\sqrt{\lambda_T/T}$ and obtain that:
\begin{equation}\label{l2}
\begin{array}{ll}
\frac{T}{4\lambda_T}\left(-2a v_{4,T} - 2 b v_{3,T}+ \beta_T \alpha_T - ab \right)=& \displaystyle -\frac{b}{2\left(a-2\right)} \left(\left(u_1+\rho u_3\right)^2 + \left(1- \rho^2\right) u_3^2\right)\\
&\displaystyle -\frac{a}{2b}\left(\left(u_2+\rho u_4\right)^2+ \left(1-\rho^2\right) u_4^2\right)\\
&\displaystyle + \left(u_1+\rho u_3\right)\left(u_2 +\rho u_4\right) + o\left(1\right).
\end{array}
\end{equation}
Thus, for $T$ going to infinity, the limit value of $T \mathcal{C}_T / \lambda_T$ is the sum of the limits of \eqref{l1} and \eqref{l2}.
Before concluding, we will now show that
\begin{equation}\label{l3}
\lim_{T \rightarrow +\infty} \lambda_T^{-1}  \log \dE_{\gamma_T, \delta_T}^{\alpha_T,\beta_T}\left( \exp \left(w_{1,T} X_T\right)X_T ^{w_{2,T}}\right) = 0
\end{equation}
The density function of the solution $X_T$ associated with parameters $\alpha_T$ and $\beta_T$ and initial point $x_0$ is given, for any positive real $y$, by
\begin{equation*}
f(y)= K_T \exp \left(-y/2\right) y^{(\alpha_T-2)/4} I_{(\alpha_{T}-2)/2}\left(\sqrt{y \xi_T}\right)
\end{equation*} 
where $I_{\nu}$ is the modified Bessel function of the first kind and $\xi_T$ and $K_T$ are two constants respectively given by
\begin{equation*}
\xi_T= -\frac{x_0 \beta_T}{e^{-\beta_T T}-1} \hspace{1cm} \text{and} \hspace{1cm} K_T= \frac{e^{-\xi_T/2}}{2 \xi_T^{(\alpha_T-2)/4}},
\end{equation*}
see for instance \cite{LL}.
Thus, using formulas 6.643(2) and 9.220(2) of \cite{GR}, we compute the expectation in the last term of \eqref{EE4} as follows
\begin{equation*}
\begin{array}{lcl}
\dE_{\gamma_T, \delta_T}^{\alpha_T,\beta_T}\left( \exp \left(w_{1,T} X_T\right)X_T ^{w_{2,T}}\right)& = & \int_0^{+\infty} \exp \left(w_{1,T} y\right) y ^{w_{2,T}} f(y) \, \mathrm{d}y\\
&=&\displaystyle \frac{\Gamma(w_{2,T} + \alpha_T/2)}{\Gamma(\alpha_T/2)} \,   2^{(\alpha_T+2)/4} \left(1/2- w_{1,T}\right)^{-(2 w_{2,T}+\alpha_T-1)/2}\\
& & \times \, \displaystyle e^{-\xi_T/2}\, {}_1F_1\left(w_{2,T} + \alpha_T/2 , \alpha_T/2, \xi_T/(2-4 w_{1,T}) \right)
\end{array}
\end{equation*}
where ${}_1F_1$ is the degenerate hypergeometric function (see \cite{GR}). As we want to compute the limit of the logarithm of this expectation, the obtained product becomes a sum and we investigate the limit of each term separately. For $T$ going to infinity, $\alpha_T$ converges to $a$ and $\beta_T$ to $b$ whereas $\xi_T$, $w_{1,T}$ and $w_{2,T}$ vanish. Thus, we obtain the four following limits
\begin{equation*}
\lim_{T \rightarrow +\infty} \lambda_T^{-1} \log \frac{\Gamma(w_{2,T} + \alpha_T/2)}{\Gamma(\alpha_T/2)} = \lim_{T \rightarrow +\infty} \lambda_T^{-1} \log\frac{\Gamma(a/2)}{\Gamma(a/2)}=0 ,
\end{equation*}

\begin{equation*}
\lim_{T \rightarrow +\infty} \lambda_T^{-1} \frac{\alpha_T+2}{4} \log 2 =0,
\end{equation*}

\begin{equation*}
\lim_{T \rightarrow +\infty} \lambda_T^{-1} ( w_{2,T}+\alpha_T/2-1/2) \log \left(1/2- w_{1,T}\right)= \lim_{T \rightarrow +\infty} \lambda_T^{-1}(a/2-1/2) \log(1/2) =0,
\end{equation*}

\begin{equation*}
\lim_{T \rightarrow +\infty} \lambda_T^{-1} \xi_T/2 = 0.
\end{equation*}
Furthermore, 
\begin{equation*}
\lim_{T \rightarrow +\infty} {}_1F_1\left(w_{2,T} + \alpha_T/2 , \alpha_T/2, \xi_T/(2-4 w_{1,T}) \right)= {}_1F_1\left(a/2 , a/2, 0 \right) = 1
\end{equation*}
which, combined with the previous limits, leads to \eqref{l3}.
 We conclude from the conjunction of \eqref{Lim1}, \eqref{l1}, \eqref{l2} and \eqref{l3}, that
\begin{equation}\label{ncgf_heston}\begin{array}{ll}
\displaystyle \lim_{T \to \infty} \lambda_T^{-1}\log \dE_T\left(u\right)=& \displaystyle  -\frac{b}{2\left(a-2\right)}\left(u_1^2 + 2 \rho u_1 u_3+ u_3^2\right)
-\frac{a}{2b}\left(u_2^2+ 2 \rho u_2 u_4 + u_4^2\right)\\
&\displaystyle + u_1 u_2 + u_3 u_4 + \rho \left(u_1 u_4 + u_2 u_3\right),
\end{array}
\end{equation}
which rewrites as a product of matrices given by \eqref{limncgf2}.

\begin{cor}
Let $\Lambda_T$ be the normalized cumulant generating function  of $((\lambda_T T)^{-1/2}M_T)$ given, for all $v \in \dR^2$, by
\begin{equation*}
\Lambda_T\left(v\right)= \frac{1}{\lambda_T} \log \dE \left[ e^{\sqrt{\lambda_T/T} \, \left\langle v, M_T\right\rangle}\right].
\end{equation*} Its pointwise limit $\Lambda$ satisfies
 \begin{equation}
\Lambda(v)=\lim_{T \to +\infty} \Lambda_T \left(v\right)= \frac{1}{2} \,  v^{\intercal} \,  \Sigma \, v, 
\end{equation} where the matrix $\Sigma$ was previously given in \eqref{defSigma}.
\end{cor}

\begin{proof}
One can observe that
\begin{equation*}
\Lambda_T\left(v\right)=  \frac{1}{\lambda_T} E_T(v^{\intercal},0,0),
\end{equation*} where $E_T$ is defined and computed in the previous proof.
Taking $u_3=u_4=0$ in \eqref{ncgf_heston}, we obtain the announced result.
\end{proof}

\begin{cor}
Let $L_T$ be the normalized cumulant generating function  of $((\lambda_T T)^{-1/2}N_T)$ given, for all $v \in \dR^2$, by
\begin{equation*}
L_T\left(v\right)= \frac{1}{\lambda_T} \log \dE \left[ e^{\sqrt{\lambda_T/T} \, \left\langle v, N_T\right\rangle}\right].
\end{equation*} Its pointwise limit $L$ satisfies
 \begin{equation}
L(v)=\lim_{T \to +\infty} L_T \left(v\right)= \frac{1}{2} \,  v^{\intercal} \,  \Sigma \, v, 
\end{equation} where the matrix $\Sigma$ was previously given in \eqref{defSigma}.
\end{cor}

\begin{proof}
As for the previous corollary, one can observe that
\begin{equation*}
L_T\left(v\right)= \frac{1}{\lambda_T} E_T(0,0,v^{\intercal}),
\end{equation*}
Taking $u_1=u_2=0$ in \eqref{ncgf_heston}, we obtain the announced result.
\end{proof}

\section*{Appendix C: Proof of the exponential convergence for $S_T$ and $\Sigma_T$}
\renewcommand{\thesection}{\Alph{section}}
\renewcommand{\theequation}{\thesection.\arabic{equation}}
\setcounter{section}{3}
\setcounter{equation}{0}
We have established, in Lemma 3.1 of \cite{dRdC1}, that the couple $(S_T,\Sigma_T)$ satisfies an LDP with speed $T$ and good rate function $I$ given for any $(x,y) \in \mathbb{R}^2$ by
\begin{equation}\label{tauxI}
I(x,y) =\left\lbrace \begin{array}{ll}
  \displaystyle \frac{y}{2(xy-1)}+\frac{b^2}{8}x+\frac{(a-2)^2}{8}y + \frac{ab}{4}  & \text{if } x>0, y>0 \text{ and } xy -1 > 0,\\
  \displaystyle +\infty & \text{otherwise. } 
\end{array}   \right.
\end{equation}
We denote by $\cD_I$ the domain of $\dR^2$ where $I$ is finite. Using the contraction principle recalled in Lemma \ref{CP}, we show that the sequence $(\Sigma_T/V_T, S_T/V_T)$ satisfies an LDP with good rate function $J$ given, for all $(z,t) \in \mathbb{R}^2$, by 
\begin{equation*}
J(z,t)= \inf \left\lbrace I(x,y) | (x,y) \in \cD_I, g(x,y)=(z,t) \right\rbrace \hspace{1cm} 
\end{equation*}
where $g(x,y)=(y(xy-1)^{-1},x(xy-1)^{-1})$,
and the infimum over the empty set is equal to infinity.
One can observe that, if $z \leq 0$ or $t \leq 0$, $\left(x,y\right)$ is not inside $\cD_I$, which means that the finitude domain $\cD_J$ of $J$ is included inside $\dR^{+}_{*} \times \dR^{+}_{*}$.
For any $z>0$ and $t>0$, the condition $(z,t)=g(x,y)$ leads to 
\begin{equation}\label{condxy}
y = z (xy-1) \hspace{1cm} \text{and} \hspace{1cm} x= tz^{-1}y.
\end{equation}
Combining both, we obtain that $y$ satisfies $ty^{2}-y-z=0$. Only one solution is strictly positive, it is given by 
\begin{equation}\label{ystar}
y^{*}= \frac{1+\sqrt{1+4tz}}{2t}.
\end{equation}
We deduce from the second equality in \eqref{condxy} that the only possible value for $x$ is given by 
\begin{equation}\label{xstar}
x^{*}= \frac{1+\sqrt{1+4tz}}{2z}.
\end{equation}
Replacing \eqref{ystar} and \eqref{xstar} into \eqref{tauxI}, we obtain that for any $(z,t) \in \dR^{+}_{*} \times \dR^{+}_{*}$,
\begin{equation}\label{tauxJ}
J(z,t)= I(x^{*},y^{*})= \frac{z}{2}+ \left(\frac{b^2}{16 z} +\frac{(a-2)^2}{16 t}\right) \left(1+\sqrt{1+4tz}\right)+ \frac{ab}{4}.
\end{equation}
The function $I$ is positive and only vanishes at point $(-ab^{-1}, -b(a-2)^{-1})$, see \cite{dRdC1}. Thus, \eqref{tauxJ} implies that $J$ is positive and vanishes for $(x^{*},y^{*})= (-ab^{-1}, -b(a-2)^{-1})$. As, $(x^{*},y^{*})$ satisfy $(z,t)=g(x^{*},y^{*})$, we obtain that $J$ only vanishes at point 
$$(z_0,t_0)=  \left( \frac{x^{*}}{x^{*}y^{*}-1}, \frac{y^{*}}{x^{*}y^{*}-1}\right)=\left(-\frac{a(a-2)}{2b},-\frac{b}{2}\right).$$
Besides, applying the contraction principle again, we show that $p^{\eta}$, given in \eqref{Peta}, satisfies the following upper bound
\begin{equation}
p^{\eta} \leq - \inf \left\lbrace J(z,t) \mid (z,t)\notin \mathcal{B}\left((z_0,t_0),\eta^2\right) \right\rbrace.
\end{equation}
We want to prove that $p^{\eta}$ is strictly negative. 
 One can observe that the function $J$ is coercive. Indeed, let $K$ be the compact subset of $(\dR_{+}^{*})^2$ given by $K=[\varepsilon,A] \times [\xi, B]$ with
 \begin{equation}
 \varepsilon= \frac{b^2}{4(4-ab)}, \, \, \, \, A= \frac{4-ab}{2}, \, \,\, \,  \xi= \frac{(a-2)^2}{4(4-ab)} \, \, \, \, \text{ and } \, \, \, \,  B= \frac{4 (4-ab)^5}{b^6},
 \end{equation}
we easily show that 
$$\forall (z,t) \notin K,  \hspace{0.5cm} J(z,t) \geq 1.$$ Consequently, the infimum of $J$ over $(\dR_{+}^{*})^2 \setminus \mathcal{B}\left((z_0,t_0),\eta^2\right)$ reduces to the infimum over the compact subset $K \setminus \mathcal{B}\left((z_0,t_0),\eta^2\right)$.  As  $J$ is a continuous function, this infimum is reached for some $(z^{*},t^{*})$. As $J$ is positive and only vanishes at point $(z_0,t_0)$ which does not belong to $K \setminus \mathcal{B}\left((z_0,t_0),\eta^2\right)$, we can conclude that
  \begin{equation*}
 \inf \left\lbrace J(z,t) \mid (z,t)\notin \mathcal{B}\left((z_0,t_0),\eta^2\right) \right\rbrace = J(z^{*},t^{*}) >0,
\end{equation*}
which is exactly the result that we wanted to prove.

\bibliographystyle{acm}
\bibliography{biblio}

\end{document}